\documentclass[a4paper,12pt]{article}
\usepackage[utf8]{inputenc}
\usepackage[T1]{fontenc}
\usepackage{amsfonts,amsmath,amssymb,amsthm}
\theoremstyle{definition}

\newtheorem{conjecture}{Conjecture}
\theoremstyle{plain}

\newtheorem{proposition}{Proposition}
\newtheorem{theorem}{Theorem}
\usepackage{graphicx}
\usepackage{caption}
\usepackage{pgf,tikz,circuitikz,subfig}
\usetikzlibrary{arrows,shapes,positioning}
\usepackage{booktabs,multirow,tabularx}
\usepackage{hyperref}
\hypersetup{colorlinks,%
	citecolor={red}, 
	urlcolor={blue},
	linkcolor={blue},
	breaklinks={true},
	pagebackref={true},
	hyperindex={true},
}
\usepackage{url}
\usepackage{cite}
\usepackage[left=2cm,right=2cm,top=2cm,bottom=2cm]{geometry}
\usepackage[ruled]{algorithm}
\usepackage{algorithmic}
\usepackage{newtxtext}
\usepackage{courier}
\usepackage{enumitem}
\newlist{abbrv}{itemize}{1}
\setlist[abbrv,1]{label=,labelwidth=0.9in,align=parleft,noitemsep,leftmargin=!}
\newcommand{\R}{\mathbb{R}}

\newcommand{\rv}[1]{\boldsymbol{#1}}

\newcommand{\ub}[1]{\overline{#1}}
\newcommand{\lb}[1]{\underline{#1}}
\newcommand{\geo}[1]{\mathtt{#1}}

\DeclareMathOperator{\subj}{s.t.}

\title{Largest small polygons: A sequential convex optimization approach}
\author{Christian Bingane\thanks{D\'{e}partement de math\'{e}matiques et de g\'{e}nie industriel, Polytechnique Montr\'{e}al, Montreal, Quebec, Canada, H3C~3A7. Email: \url{christian.bingane@polymtl.ca}}}

\begin{document}
\maketitle
\begin{abstract}
A small polygon is a polygon of unit diameter. The maximal area of a small polygon with $n=2m$ vertices is not known when $m\ge 7$. Finding the largest small $n$-gon for a given number $n\ge 3$ can be formulated as a nonconvex quadratically constrained quadratic optimization problem. We propose to solve this problem with a sequential convex optimization approach, which is an ascent algorithm guaranteeing convergence to a locally optimal solution. Numerical experiments on polygons with up to $n=128$ sides suggest that the optimal solutions obtained are near-global. Indeed, for even $6 \le n \le 12$, the algorithm proposed in this work converges to known global optimal solutions found in the literature.
\end{abstract}

\section{Introduction}
The {\em diameter} of a polygon is the largest Euclidean distance between pairs of its vertices. A polygon is said to be {\em small} if its diameter equals one. For a given integer $n \ge 3$, the maximal area problem consists in finding a small $n$-gon with the largest area. The problem was first investigated by Reinhardt~\cite{reinhardt1922} in 1922. He proved that
\begin{itemize}
\item when $n$ is odd, the regular small $n$-gon is the unique optimal solution;
\item when $n=4$, there are infinitely many optimal solutions, including the small square;
\item when $n \ge 6$ is even, the regular small $n$-gon is not optimal.
\end{itemize}

When $n \ge 6$ is even, the maximal area problem is solved for $n \le 12$. In 1961, Bieri~\cite{bieri1961} found the largest small $6$-gon, assuming the existence of an axis of symmetry. In 1975, Graham~\cite{graham1975} independently constructed the same $6$-gon, represented in Figure~\ref{figure:6gon:U6}. In 2002, Audet, Hansen, Messine, and Xiong~\cite{audet2002} combined Graham's strategy with global optimization methods to find the largest small $8$-gon, illustrated in Figure~\ref{figure:8gon:U8}. In 2013, Henrion and Messine~\cite{henrion2013} found the largest small $10$- and $12$-gons by also solving globally a nonconvex quadratically constrained quadratic optimization problem. They also found the largest small axially symmetrical $14$- and $16$-gons. In 2017, Audet~\cite{audet2017} showed that the regular small polygon has the maximal area among all equilateral small polygons. In 2021, Audet, Hansen, and Svrtan~\cite{audet2021} determined analytically the largest small axially symmetrical $8$-gon.

The diameter graph of a small polygon is the graph with the vertices of the polygon, and an edge between two vertices exists only if the distance between these vertices equals one. Graham~\cite{graham1975} conjectured that, for even $n \ge 6$, the diameter graph of a small $n$-gon with maximal area has a cycle of length $n-1$ and one additional edge from the remaining vertex. The case $n=6$ was proven by Graham himself~\cite{graham1975} and the case $n=8$ by Audet, Hansen, Messine, and Xiong~\cite{audet2002}. In 2007, Foster and Szabo~\cite{foster2007} proved Graham's conjecture for all even $n \ge 6$. Figure~\ref{figure:4gon}, Figure~\ref{figure:6gon}, and Figure~\ref{figure:8gon} show diameter graphs of some small polygons. The solid lines illustrate pairs of vertices which are unit distance apart.

In addition to exact results and bounds, uncertified largest small polygons have been obtained both by metaheurisitics and nonlinear optimization. Assuming Graham's conjecture and the existence of an axis of symmetry, Mossinghoff~\cite{mossinghoff2006b} in 2006 constructed large small $n$-gons for even $6 \le n \le 20$. In 2020, using a formulation based on polar coordinates, Pint{\'e}r~\cite{pinter2020} presented numerical solutions estimates of the maximal area for even $6 \le n \le 80$. The polar coordinates-based formulation was recently used by Pint{\'e}r, Kampas, and Castillo~\cite{pinter2022} to obtain estimates of the maximal area for even $6 \le n \le 1000$.

The maximal area problem can be formulated as a nonconvex quadratically constrained quadratic optimization problem. In this work, we propose to solve it with a sequential convex optimization approach, also known as the concave-convex procedure~\cite{marks1978,sriperumbudur2009}. This approach is an ascent algorithm guaranteeing convergence to a locally optimal solution. Numerical experiments on polygons with up to $n=128$ sides suggest that the optimal solutions obtained are near-global. Indeed, without assuming Graham's conjecture nor the existence of an axis of symmetry in our quadratic formulation, optimal $n$-gons obtained with the algorithm proposed in this work verify both conditions within the limit of numerical computations. Moreover, for even $6 \le n \le 12$, this algorithm converges to known global optimal solutions. The algorithm is implemented as a package, OPTIGON, which is available on GitHub~\cite{optigon}.

The remainder of this paper is organized as follows. In Section~\ref{sec:ngon}, we recall principal results on largest small polygons. Section~\ref{sec:nqcqo} presents the quadratic formulation of the maximal area problem and the sequential convex optimization approach to solve it. We report in Section~\ref{sec:results} computational results. Section~\ref{sec:conclusion} concludes the paper.

\begin{figure}[h]
	\centering
	\subfloat[$(\geo{R}_4,0.5)$]{
		\begin{tikzpicture}[scale=4]
		\draw[dashed] (0,0) -- (0.5000,0.5000) -- (0,1) -- (-0.5000,0.5000) -- cycle;
		\draw (0,0) -- (0,1);
		\draw (0.5000,0.5000) -- (-0.5000,0.5000);
		\end{tikzpicture}
	}
	\subfloat[$(\geo{R}_3^+,0.5)$]{
		\begin{tikzpicture}[scale=4]
		\draw[dashed] (0.5000,0.8660) -- (0,1) -- (-0.5000,0.8660);
		\draw (0,1) -- (0,0) -- (0.5000,0.8660) -- (-0.5000,0.8660) -- (0,0);
		\end{tikzpicture}
	}
	\caption{Two small $4$-gons $(\geo{P}_4,A(\geo{P}_4))$}
	\label{figure:4gon}
\end{figure}
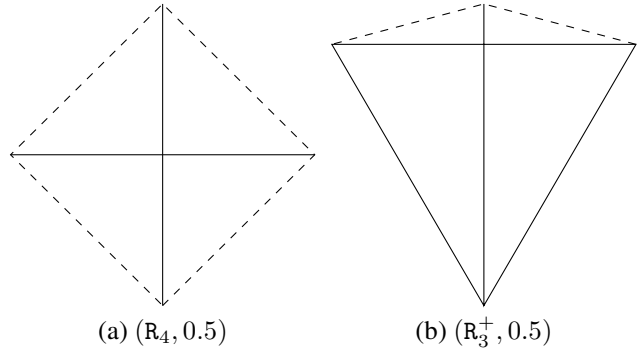

\begin{figure}[h]
	\centering
	\subfloat[$(\geo{R}_6,0.649519)$]{
		\begin{tikzpicture}[scale=4]
		\draw[dashed] (0,0) -- (0.4330,0.2500) -- (0.4330,0.7500) -- (0,1) -- (-0.4330,0.7500) -- (-0.4330,0.2500) -- cycle;
		\draw (0,0) -- (0,1);
		\draw (0.4330,0.2500) -- (-0.4330,0.7500);
		\draw (0.4330,0.7500) -- (-0.4330,0.2500);
		\end{tikzpicture}
	}
	\subfloat[$(\geo{R}_5^+,0.672288)$]{
		\begin{tikzpicture}[scale=4]
		\draw[dashed] (0,0) -- (0.5000,0.3633) -- (0.3090,0.9511) -- (0,1) -- (-0.3090,0.9511) -- (-0.5000,0.3633) -- cycle;
		\draw (0,1) -- (0,0) -- (0.3090,0.9511) -- (-0.5000,0.3633) -- (0.5000,0.3633) -- (-0.3090,0.9511) -- (0,0);
		\end{tikzpicture}
	}
	\subfloat[$(\geo{P}_6^*,0.674981)$]{
		\begin{tikzpicture}[scale=4]
		\draw[dashed] (0,0) -- (0.5000,0.4024) -- (0.3438,0.9391) -- (0,1) -- (-0.3438,0.9391) -- (-0.5000,0.4024) -- cycle;
		\draw (0,1) -- (0,0) -- (0.3438,0.9391) -- (-0.5000,0.4024) -- (0.5000,0.4024) -- (-0.3438,0.9391) -- (0,0);
		\end{tikzpicture}
		\label{figure:6gon:U6}
	}
	\caption{Three small $6$-gons $(\geo{P}_6,A(\geo{P}_6))$}
	\label{figure:6gon}
\end{figure}
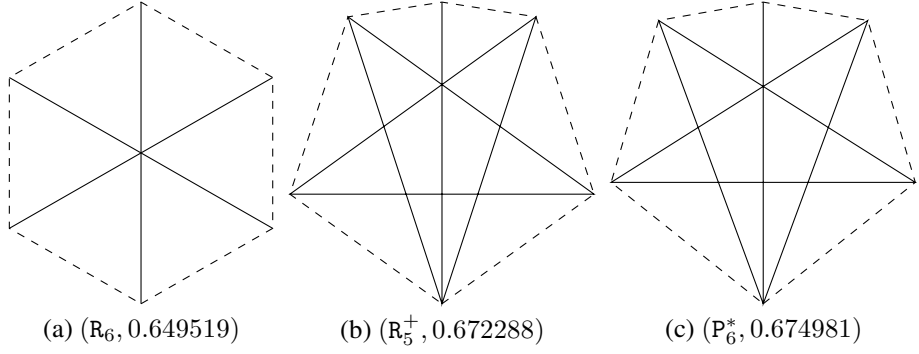

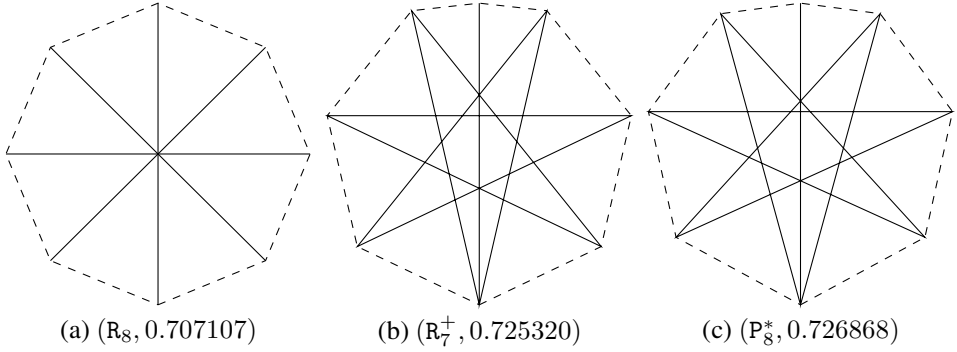
\begin{figure}[h]
	\centering
	\subfloat[$(\geo{R}_8,0.707107)$]{
		\begin{tikzpicture}[scale=4]
		\draw[dashed] (0,0) -- (0.3536,0.1464) -- (0.5000,0.5000) -- (0.3536,0.8536) -- (0,1) -- (-0.3536,0.8536) -- (-0.5000,0.5000) -- (-0.3536,0.1464) -- cycle;
		\draw (0,0) -- (0,1);
		\draw (0.3536,0.1464) -- (-0.3536,0.8536);
		\draw (0.5000,0.5000) -- (-0.5000,0.5000);
		\draw (0.3536,0.8536) -- (-0.3536,0.1464);
		\end{tikzpicture}
	}
	\subfloat[$(\geo{R}_7^+,0.725320)$]{
		\begin{tikzpicture}[scale=4]
		\draw[dashed] (0,0) -- (0.4010,0.1931) -- (0.5000,0.6270) -- (0.2225,0.9749) -- (0,1) -- (-0.2225,0.9749) -- (-0.5000,0.6270) -- (-0.4010,0.1931) -- cycle;
		\draw (0,1) -- (0,0) -- (0.2225,0.9749) -- (-0.4010,0.1931) -- (0.5000,0.6270) -- (-0.5000,0.6270) -- (0.4010,0.1931) -- (-0.2225,0.9749) -- (0,0);
		\end{tikzpicture}
	}
	\subfloat[$(\geo{P}_8^*,0.726868)$]{
		\begin{tikzpicture}[scale=4]
		\draw[dashed] (0,0) -- (0.4091,0.2238) -- (0.5000,0.6404) -- (0.2621,0.9650) -- (0,1) -- (-0.2621,0.9650) -- (-0.5000,0.6404) -- (-0.4091,0.2238) -- cycle;
		\draw (0,1) -- (0,0) -- (0.2621,0.9650) -- (-0.4091,0.2238) -- (0.5000,0.6404) -- (-0.5000,0.6404) -- (0.4091,0.2238) -- (-0.2621,0.9650) -- (0,0);
		\end{tikzpicture}
		\label{figure:8gon:U8}
	}
	\caption{Three small $8$-gons $(\geo{P}_8,A(\geo{P}_8))$}
	\label{figure:8gon}
\end{figure}

\section{Largest small polygons}\label{sec:ngon}
Let $A(\geo{P})$ denote the area of a polygon $\geo{P}$. Let $\geo{R}_n$ denote the regular small $n$-gon. We have
\[
A(\geo{R}_n) =
\begin{cases}
\frac{n}{2}\left(\sin \frac{\pi}{n} - \tan \frac{\pi}{2n}\right) &\text{if $n$ is odd,}\\
\frac{n}{8}\sin \frac{2\pi}{n} &\text{if $n$ is even.}\\
\end{cases}
\]
We remark that $A(\geo{R}_n) < A(\geo{R}_{n-1})$ for all even $n\ge 6$~\cite{audet2009a}. This suggests that $\geo{R}_n$ does not have maximum area for any even $n\ge 6$. Indeed, when $n$ is even, we can construct a small $n$-gon with a larger area than $\geo{R}_n$ by adding a vertex at distance $1$ along the mediatrix of an angle in $\geo{R}_{n-1}$. We denote this $n$-gon by $\geo{R}_{n-1}^+$ and we have
\[
A(\geo{R}_{n-1}^+) = \frac{n-1}{2} \left(\sin \frac{\pi}{n-1} - \tan \frac{\pi}{2n-2}\right) + \sin \frac{\pi}{2n-2} - \frac{1}{2}\sin \frac{\pi}{n-1}.
\]

\begin{theorem}[Reinhardt~\cite{reinhardt1922}]
For all $n \ge 3$, let $A_n^*$ denote the maximal area among all small $n$-gons and let $\ub{A}_n := \frac{n}{2}\left(\sin \frac{\pi}{n} - \tan \frac{\pi}{2n}\right)$.
\begin{itemize}
\item When $n$ is odd, $A_n^* = \ub{A}_n$ is only achieved by $\geo{R}_n$.
\item $A_4^* = 1/2 < \ub{A}_4$ is achieved by infinitely many $4$-gons, including $\geo{R}_4$ and~$\geo{R}_3^+$ illustrated in Figure~\ref{figure:4gon}.
\item When $n\ge 6$ is even, $A(\geo{R}_n) < A_n^* < \ub{A}_n$.
\end{itemize}
\end{theorem}

When $n\ge 6$ is even, the maximal area~$A_n^*$ is known for even $n \le 12$. Using geometric arguments, Graham~\cite{graham1975} determined analytically the largest small $6$-gon, represented in Figure~\ref{figure:6gon:U6}. Its area $A_6^* \approx 0.674981$ is about $3.92\%$ larger than $A(\geo{R}_6) = 3\sqrt{3}/8$. The approach of Graham, combined with methods of global optimization, has been followed by~\cite{audet2002} to determine the largest small $8$-gon, represented in Figure~\ref{figure:8gon:U8}. Its area $A_8^* \approx 0.726868$ is about $2.79\%$ larger than $A(\geo{R}_8) = \sqrt{2}/2$. Henrion and Messine~\cite{henrion2013} found that $A_{10}^* \approx 0.749137$ and $A_{12}^* \approx 0.760730$.

For all even $n\ge 6$, let $\geo{P}_n^*$ denote an optimal small $n$-gon.
\begin{theorem}[Graham~\cite{graham1975}, Foster and Szabo~\cite{foster2007}]
\label{thm:area:diam}
For even $n \ge 6$, the diameter graph of $\geo{P}_n^*$ has a cycle of length $n-1$ and one additional edge from the remaining vertex.
\end{theorem}
\begin{conjecture}
\label{thm:area:sym}
For even $n \ge 6$, $\geo{P}_n^*$ has an axis of symmetry corresponding to the pending edge in its diameter graph.
\end{conjecture}
From Theorem~\ref{thm:area:diam}, we note that $\geo{R}_{n-1}^+$ has the same diameter graph as the largest small $n$-gon $\geo{P}_n^*$. Conjecture~\ref{thm:area:sym} is only proven for $n=6$ and this is due to Yuan~\cite{yuan2004}. However, the largest small polygons obtained by~\cite{audet2002} and~\cite{henrion2013} are a further evidence that the conjecture may be true.

\section{Nonconvex quadratically constrained quadratic optimization} \label{sec:nqcqo}
We use cartesian coordinates to describe an $n$-gon $\geo{P}_n$, assuming that a vertex $\geo{v}_i$, $i=0,1,\ldots,n-1$, is positioned at abscissa $x_i$ and ordinate $y_i$. Placing the vertex $\geo{v}_0$ at the origin, we set $x_0 = y_0 = 0$. We also assume that the $n$-gon $\geo{P}_n$ is in the half-plane $y\ge 0$ and the vertices $\geo{v}_i$, $i=1,2,\ldots,n-1$, are arranged in a counterclockwise order as illustrated in Figure~\ref{figure:model}, i.e., $x_iy_{i+1} \ge y_ix_{i+1}$ for all $i=1,2,\ldots,n-2$. The maximal area problem can be formulated as follows
\begin{subequations}\label{eq:ngon:area}
	\begin{align}
	\max_{\rv{x},\rv{y},\rv{u}} \quad & \sum_{i=1}^{n-2} u_i\\
	\subj \quad & (x_j - x_i)^2 + (y_j - y_i)^2 \le 1 &\forall 1\le i < j \le n-1,\label{eq:ngon:d}\\
	& x_i^2 + y_i^2 \le 1 &\forall 1 \le i \le n-1,\label{eq:ngon:r}\\
	& y_i \ge 0 &\forall 1 \le i \le n-1,\label{eq:ngon:y}\\
	& 2u_i \le x_iy_{i+1} - y_ix_{i+1} &\forall 1 \le i \le n-2,\label{eq:ngon:u}\\
	& u_i \ge 0 &\forall 1 \le i \le n-2.
	\end{align}
\end{subequations}
At optimality, for all $i=1,2,\ldots,n-2$, $u_i = (x_iy_{i+1} - y_ix_{i+1})/2$, which corresponds to the area of the triangle $\geo{v}_0\geo{v}_i\geo{v}_{i+1}$.
It is important to note that, unlike what was done in~\cite{audet2002,henrion2013}, this formulation does not make the assumption of Graham's conjecture, nor of the existence of an axis of symmetry.

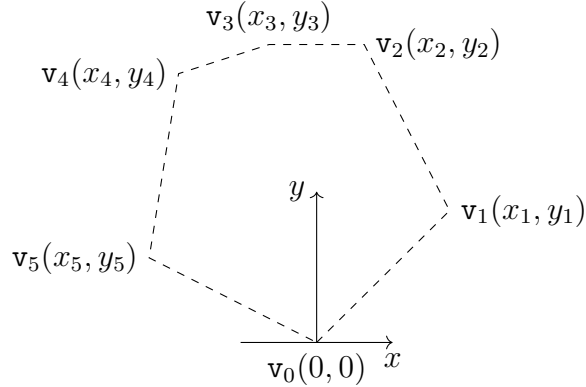
\begin{figure}[h]
\centering
\begin{tikzpicture}[scale=4]
\draw[dashed] (0,0) node[below]{$\geo{v}_0(0,0)$} -- (0.4370,0.4370) node[right]{$\geo{v}_1(x_1,y_1)$} -- (0.1564,0.9877) node[right]{$\geo{v}_2(x_2,y_2)$} -- (-0.1564,0.9877) node[above]{$\geo{v}_3(x_3,y_3)$} -- (-0.4540,0.8911) node[left]{$\geo{v}_4(x_4,y_4)$} -- (-0.5507,0.2806) node[left]{$\geo{v}_5(x_5,y_5)$} -- cycle;
\draw[->] (-0.25,0)--(0.25,0)node[below]{$x$};
\draw[->] (0,0)--(0,0.5)node[left]{$y$};
\end{tikzpicture}
\caption{Definition of variables: Case of $n=6$ vertices}
\label{figure:model}
\end{figure}

Problem~\eqref{eq:ngon:area} is a nonconvex quadratically constrained quadratic optimization problem and can be reformulated as a difference-of-convex optimization (DCO) problem of the form
\begin{subequations}\label{eq:dco}
	\begin{align}
	\max_{\rv{z}} \quad & g_0(\rv{z}) - h_0(\rv{z})\\
	\subj \quad& g_i(\rv{z}) - h_i(\rv{z}) \ge 0 &\forall 1 \le i \le m,
	\end{align}
\end{subequations}
where $g_0,\ldots,g_m$ and $h_0,\ldots,h_m$ are convex quadratic functions. We note that the feasible set
\[
\Omega := \{\rv{z} \colon g_i(\rv{z}) - h_i(\rv{z}) \ge 0, i =1,2,\ldots,m\}
\]
is compact with a nonempty interior, which implies that $g_0(\rv{z}) - h_0(\rv{z}) < \infty$ for all $\rv{z} \in \Omega$.

For a fixed $\rv{c}$, we have $\lb{g}_i(\rv{z};\rv{c}) := g_i(\rv{c}) + \nabla g_i(\rv{c})^T (\rv{z} - \rv{c}) \le g_i(\rv{z})$ for all $i=0,1,\ldots,m$. Then the following problem
\begin{subequations}\label{eq:dcocvx}
	\begin{align}
	\max_{\rv{z}} \quad & \lb{g}_0(\rv{z};\rv{c}) - h_0(\rv{z})\\
	\subj \quad& \lb{g}_i(\rv{z};\rv{c}) - h_i(\rv{z}) \ge 0 &\forall 1 \le i \le m
	\end{align}
\end{subequations}
is a convex restriction of the DCO problem~\eqref{eq:dco} as stated by Proposition~\ref{thm:cvxrestr}. Constraint~\eqref{eq:ngon:u} is equivalent to
\[
(x_i-y_{i+1})^2+(y_i+x_{i+1})^2+8u_i \le (x_i+y_{i+1})^2+(y_i-x_{i+1})^2
\]
for all $i=1,2,\ldots,n-2$. For a fixed $(\rv{a},\rv{b}) \in \R^{n-1} \times \R^{n-1}$, if we replace~\eqref{eq:ngon:u} in~\eqref{eq:ngon:area} by
\[
(x_i-y_{i+1})^2+(y_i+x_{i+1})^2+8u_i \le 2(a_i+b_{i+1})(x_i+y_{i+1})-(a_i+b_{i+1})^2 + 2(b_i-a_{i+1})(y_i-x_{i+1})-(b_i-a_{i+1})^2
\]
for all $i=1,2,\ldots,n-2$, we obtain a convex restriction of the maximal area problem.

\begin{proposition}\label{thm:cvxrestr}
If $\rv{z}$ is a feasible solution of~\eqref{eq:dcocvx} then $\rv{z}$ is a feasible solution of~\eqref{eq:dco}.
\end{proposition}
\begin{proof}
Let $\rv{z}$ be a feasible solution of~\eqref{eq:dcocvx}, i.e., $\lb{g}_i(\rv{z};\rv{c}) - h_i(\rv{z}) \ge 0$ for all $i=1,2,\ldots,m$. Then $g_i(\rv{z}) - h_i(\rv{z}) \ge \lb{g}_i(\rv{z};\rv{c}) - h_i(\rv{z}) \ge 0$ for all $i=1,2,\ldots,m$. Thus, $\rv{z}$ is a feasible solution of~\eqref{eq:dco}.
\end{proof}

\begin{proposition}\label{thm:ascent}
If $\rv{c}$ is a feasible solution of~\eqref{eq:dco} then \eqref{eq:dcocvx} is a feasible problem. Moreover, if $\rv{z}^*$ is an optimal solution of~\eqref{eq:dcocvx} then $g_0(\rv{c}) - h_0(\rv{c}) \le g_0(\rv{z}^*) - h_0(\rv{z}^*)$.
\end{proposition}
\begin{proof}
Let $\rv{c}$ be a feasible solution of~\eqref{eq:dco}, i.e., $g_i(\rv{c}) - h_i(\rv{c}) \ge 0$ for all $i=1,2,\ldots,m$. Then there exists $\rv{z} = \rv{c}$ such that $\lb{g}_i(\rv{c};\rv{c}) - h_i(\rv{c}) = g_i(\rv{c}) - h_i(\rv{c}) \ge 0$ for all $i=1,2,\ldots,m$. Thus, \eqref{eq:dcocvx} is a feasible problem. Moreover, if $\rv{z}^*$ is an optimal solution of~\eqref{eq:dcocvx}, we have $g_0(\rv{c}) - h_0(\rv{c}) = \lb{g}_0(\rv{c};\rv{c}) - h_0(\rv{c}) \le \lb{g}_0(\rv{z}^*;\rv{c}) - h_0(\rv{z}^*) \le g_0(\rv{z}^*) - h_0(\rv{z}^*)$.
\end{proof}

From Proposition~\ref{thm:ascent}, the optimal small $n$-gon $(\rv{x},\rv{y})$ obtained by solving a convex restriction of Problem~\eqref{eq:ngon:area} constructed around a small $n$-gon $(\rv{a},\rv{b})$ has a larger area than this one. Proposition~\ref{thm:local} states that if $(\rv{a},\rv{b})$ is the optimal $n$-gon of the convex restriction constructed around itself, then it is a local optimal $n$-gon for the maximal area problem.

\begin{proposition}\label{thm:local}
Let $\rv{c}$ be a feasible solution of~\eqref{eq:dco}. We suppose that $\lb{\Omega}(\rv{c}) := \{\rv{z}\colon \lb{g}_i(\rv{z};\rv{c}) - h_i(\rv{z}) \ge 0, i =1,2,\ldots,m\}$ satisfies Slater condition. If $\rv{c}$ is an optimal solution of~\eqref{eq:dcocvx} then $\rv{c}$ is a critical point of~\eqref{eq:dco}.
\end{proposition}

\begin{proof}
If $\rv{c}$ is an optimal solution of~\eqref{eq:dcocvx} then there exist $m$ scalars $\mu_1, \mu_2,\ldots,\mu_m$ such that
\[
\begin{aligned}
\nabla\lb{g}_0 (\rv{c};\rv{c}) + \sum_{i=1}^m \mu_i\nabla \lb{g}_i(\rv{c};\rv{c}) &= \nabla h_0(\rv{c}) + \sum_{i=1}^m \mu_i\nabla h_i(\rv{c}),\\
\lb{g}_i(\rv{c};\rv{c}) &\ge h_i(\rv{c}) &\forall i=1,2,\ldots,m,\\
\mu_i &\ge 0 &\forall i=1,2,\ldots,m,\\
\mu_i\lb{g}_i(\rv{c};\rv{c}) &= \mu_i h_i(\rv{c}) &\forall i=1,2,\ldots,m.
\end{aligned}
\]
Since $\lb{g}_i (\rv{c};\rv{c}) = g_i(\rv{c})$ and $\nabla \lb{g}_i (\rv{c};\rv{c}) = \nabla g_i(\rv{c})$ for all $i=0,1,\ldots,m$, we conclude that $\rv{c}$ is a critical point of~\eqref{eq:dco}.
\end{proof}

We propose to solve the DCO problem~\eqref{eq:dco} with a sequential convex optimization approach given in Algorithm~\ref{algo:ccp}, also known as concave-convex procedure. A proof of showing that a sequence $\{\rv{z}_k\}_{k=0}^\infty$ generated by Algorithm~\ref{algo:ccp} converges to a critical point $\rv{z}^*$ of the original DCO problem~\eqref{eq:dco} can be found in~\cite{marks1978,sriperumbudur2009}.

\begin{algorithm}
	\caption{Sequential convex optimization}
	\label{algo:ccp}
	\begin{algorithmic}[1]
		\STATE Initialization: choose a feasible solution $\rv{z}_0$ and a stopping criteria $\varepsilon > 0$.
		\STATE $\rv{z}_1 \in \arg\max \{\lb{g}_0(\rv{z};\rv{z}_0) - h_0(\rv{z})\colon \lb{g}_i(\rv{z};\rv{z}_0) - h_i(\rv{z}) \ge 0, i =1,2,\ldots,m\}$
		\STATE $k := 1$
		\WHILE{$\frac{\|\rv{z}_k - \rv{z}_{k-1}\|}{\|\rv{z}_k\|} > \varepsilon$}
		\STATE $\rv{z}_{k+1} \in \arg\max \{\lb{g}_0(\rv{z};\rv{z}_k) - h_0(\rv{z})\colon \lb{g}_i(\rv{z};\rv{z}_k) - h_i(\rv{z}) \ge 0, i =1,2,\ldots,m\}$
		\STATE $k := k+1$
		\ENDWHILE
	\end{algorithmic}
\end{algorithm}

\section{Computational results}\label{sec:results}
Problem~\eqref{eq:ngon:area} was solved in Julia using JuMP~0.22.1 with MOSEK~9.3.13. All the computations were carried out on an \texttt{Intel(R) Core(TM) i7-3540M CPU @ 3.00 GHz} computing platform. Algorithm~\ref{algo:ccp} was implemented as a package: OPTIGON, which is freely available on GitHub~\cite{optigon}. JuMP is a domain-specific modeling language for mathematical optimization embedded in Julia~\cite{dunning2017}.

We chose the following values as initial solution:
\[
\begin{aligned}
a_0 &= 0, &b_0 &= 0,\\
a_i &= \frac{\sin \frac{2i\pi}{n-1}}{2\cos \frac{\pi}{2n-2}} = -a_{n-i}, & b_i &= \frac{1-\cos \frac{2i\pi}{n-1}}{2\cos \frac{\pi}{2n-2}} = b_{n-i} &\forall i=1,\ldots,n/2-1,\\
a_{n/2} &= 0, &b_{n/2} &= 1,
\end{aligned}
\]
which define the $n$-gon $\geo{R}_{n-1}^+$, and the stopping criteria $\varepsilon = 10^{-5}$. At each iteration, we forced MOSEK to solve the dual problem by setting the parameter {\tt MSK\_IPAR\_INTPNT\_SOLVE\_FORM} to {\tt MSK\_SOLVE\_DUAL}. Table~\ref{table:area} shows the areas of the optimal $n$-gons $\hat{\geo{P}}_n$ for even numbers $n=6,8,\ldots,128$, along with the areas of the initial $n$-gons $\geo{R}_{n-1}^+$, the best lower bounds $\lb{A}_n$ found in the literature, and the upper bounds~$\ub{A}_n$. We also report the number $k$ of iterations of Algotithm~\ref{algo:ccp} and the total computation time for each~$n$. The results in Table~\ref{table:area} support the following keypoints:
\begin{enumerate}
\item For $6 \le n\le 12$, $\lb{A}_n - A(\hat{\geo{P}}_n) \le 10^{-8}$, i.e., Algorithm~\ref{algo:ccp} converges to the best known optimal solutions found in the literature.
\item For $32 \le n \le 80$, $\lb{A}_n < A(\geo{R}_{n-1}^+) < A(\hat{\geo{P}}_n)$, i.e., it appears that the solutions obtained by Pint{\'e}r~\cite{pinter2020} are suboptimal.
\item For all $n$, the solutions $\hat{\geo{P}}_n$ obtained with Algorithm~\ref{algo:ccp} verify, within the limit of the numerical computations, Theorem~\ref{thm:area:diam} and Conjecture~\ref{thm:area:sym}, i.e.,
\[
\begin{aligned}
x_{n/2} &=0, &y_{n/2} &=1,\\
\|\geo{v}_{n/2-1}\| &=1, &\|\geo{v}_{n/2+1}\| &=1,\\
\|\geo{v}_{i+n/2}-\geo{v}_i\| &=1, &\|\geo{v}_{i+n/2+1}-\geo{v}_i\| &=1 &\forall i=1,2,\ldots,n/2-2,\\
\|\geo{v}_{n-1}-\geo{v}_{n/2-1}\| &=1,\\
x_{n-i} &=-x_i, &y_{n-i} &= y_i &\forall i=1,2,\ldots,n/2-1.
\end{aligned}
\]
We illustrate the optimal $16$-, $32$- and $64$-gons in Figure~\ref{figure:Un}. Furthermore,  we remark that Theorem~\ref{thm:area:diam} and Conjecture~\ref{thm:area:sym} are verified by each polygon of the sequence generated by Algorithm~\ref{algo:ccp}. All $6$-gons generated by the algorithm are represented in Figure~\ref{figure:ccp:U6} and the coordinates of their vertices are given in Table~\ref{table:ccp:U6}.
\end{enumerate}


\begin{table}[t]
	\footnotesize
	\centering
	\caption{Maximal area problem}
	\label{table:area}
	\resizebox{\linewidth}{!}{
		\begin{tabular}{@{}rlll|lrr||rlll|lrr@{}}
			\toprule
			$n$ & $A(\geo{R}_{n-1}^+)$ & $\lb{A}_n$ & $\ub{A}_n$ & $A(\hat{\geo{P}}_n)$ & \# iter. $k$	& time [s]	&	$n$ & $A(\geo{R}_{n-1}^+)$ & $\lb{A}_n$ & $\ub{A}_n$ & $A(\hat{\geo{P}}_n)$ & \# iter. $k$	& time [s]	\\
			\midrule
			6	&	0.6722882584	&	0.6749814429~\cite{bieri1961,graham1975,mossinghoff2006b}	&	0.6961524227	&	0.6749814405	&	5	&	0.17	&	68	&	0.7846851407	&	0.7846139029~\cite{pinter2020}	&	0.7846997026	&	0.7846880042	&	83	&	15.02	\\
			8	&	0.7253199909	&	0.7268684828~\cite{audet2002,mossinghoff2006b}	&	0.7350842599	&	0.7268684795	&	10	&	0.26	&	70	&	0.7847256986	&	0.7846403575~\cite{pinter2020}	&	0.7847390429	&	0.7847283062	&	77	&	14.99	\\
			10	&	0.7482573378	&	0.7491373459~\cite{henrion2013,mossinghoff2006b}	&	0.7531627703	&	0.7491373453	&	16	&	0.48	&	72	&	0.7847628920	&	0.7847454020~\cite{pinter2020}	&	0.7847751508	&	0.7847652673	&	72	&	14.02	\\
			12	&	0.7601970055	&	0.7607298734~\cite{henrion2013,mossinghoff2006b}	&	0.7629992851	&	0.7607298714	&	24	&	0.70	&	74	&	0.7847970830	&	0.7845564840~\cite{pinter2020}	&	0.7848083708	&	0.7847992539	&	66	&	13.54	\\
			14	&	0.7671877750	&	0.7675310111~\cite{mossinghoff2006b}	&	0.7689359584	&	0.7675310103	&	33	&	1.00	&	76	&	0.7848285863	&	0.7847585719~\cite{pinter2020}	&	0.7848390031	&	0.7848305842	&	64	&	14.42	\\
			16	&	0.7716285345	&	0.7718613220~\cite{mossinghoff2006b}	&	0.7727913493	&	0.7718613201	&	43	&	1.36	&	78	&	0.7848576763	&	0.7845160579~\cite{pinter2020}	&	0.7848673094	&	0.7848595155	&	61	&	13.91	\\
			18	&	0.7746235089	&	0.7747881651~\cite{mossinghoff2006b}	&	0.7754356273	&	0.7747881607	&	55	&	1.79	&	80	&	0.7848845934	&	0.7848252941~\cite{pinter2020}	&	0.7848935195	&	0.7848862887	&	58	&	15.72	\\
			20	&	0.7767382147	&	0.7768587560~\cite{mossinghoff2006b}	&	0.7773275822	&	0.7768587511	&	68	&	2.57	&	82	&	0.7849095487	&	--	&	0.7849178354	&	0.7849111132	&	55	&	15.15	\\
			22	&	0.7782865351	&	0.7783773308~\cite{pinter2020}	&	0.7787276939	&	0.7783773228	&	81	&	3.05	&	84	&	0.7849327284	&	--	&	0.7849404352	&	0.7849341775	&	53	&	13.97	\\
			24	&	0.7794540033	&	0.7795240461~\cite{pinter2020}	&	0.7797927529	&	0.7795240356	&	95	&	4.06	&	86	&	0.7849542969	&	--	&	0.7849614768	&	0.7849556400	&	51	&	14.11	\\
			26	&	0.7803559816	&	0.7804111201~\cite{pinter2020}	&	0.7806217145	&	0.7804111056	&	109	&	5.28	&	88	&	0.7849744002	&	--	&	0.7849811001	&	0.7849756401	&	47	&	14.31	\\
			28	&	0.7810672517	&	0.7811114192~\cite{pinter2020}	&	0.7812795297	&	0.7811113984	&	122	&	5.34	&	90	&	0.7849931681	&	--	&	0.7849994298	&	0.7849943222	&	46	&	14.16	\\
			30	&	0.7816380102	&	0.7816739255~\cite{pinter2020}	&	0.7818102598	&	0.7816739036	&	136	&	6.83	&	92	&	0.7850107163	&	--	&	0.7850165772	&	0.7850117896	&	44	&	14.28	\\
			32	&	0.7821029651	&	0.7818946320~\cite{pinter2020}	&	0.7822446490	&	0.7821325307	&	148	&	8.39	&	94	&	0.7850271482	&	--	&	0.7850326419	&	0.7850281476	&	42	&	13.97	\\
			34	&	0.7824867354	&	0.7823103007~\cite{pinter2020}	&	0.7826046775	&	0.7825113686	&	159	&	9.28	&	96	&	0.7850425565	&	--	&	0.7850477130	&	0.7850434880	&	40	&	13.59	\\
			36	&	0.7828071755	&	0.7826513767~\cite{pinter2020}	&	0.7829063971	&	0.7828279089	&	169	&	10.20	&	98	&	0.7850570245	&	--	&	0.7850618708	&	0.7850578959	&	39	&	15.02	\\
			38	&	0.7830774889	&	0.7829526627~\cite{pinter2020}	&	0.7831617511	&	0.7830950948	&	177	&	12.30	&	100	&	0.7850706272	&	--	&	0.7850751877	&	0.7850714430	&	38	&	14.98	\\
			40	&	0.7833076096	&	0.7832011589~\cite{pinter2020}	&	0.7833797744	&	0.7833226806	&	182	&	12.85	&	102	&	0.7850834323	&	--	&	0.7850877290	&	0.7850841971	&	37	&	15.81	\\
			42	&	0.7835051276	&	0.7834135187~\cite{pinter2020}	&	0.7835674041	&	0.7835181221	&	185	&	13.68	&	104	&	0.7850955008	&	--	&	0.7850995538	&	0.7850962158	&	35	&	14.92	\\
			44	&	0.7836759223	&	0.7835966860~\cite{pinter2020}	&	0.7837300377	&	0.7836871968	&	185	&	14.53	&	106	&	0.7851068883	&	--	&	0.7851107156	&	0.7851075587	&	34	&	15.03	\\
			46	&	0.7838246055	&	0.7837554636~\cite{pinter2020}	&	0.7838719255	&	0.7838344298	&	178	&	16.10	&	108	&	0.7851176450	&	--	&	0.7851212630	&	0.7851182747	&	33	&	15.31	\\
			48	&	0.7839548353	&	0.7838942710~\cite{pinter2020}	&	0.7839964516	&	0.7839634479	&	171	&	15.91	&	110	&	0.7851278167	&	--	&	0.7851312404	&	0.7851284083	&	32	&	15.28	\\
			50	&	0.7840695435	&	0.7840161496~\cite{pinter2020}	&	0.7841063371	&	0.7840771193	&	159	&	15.57	&	112	&	0.7851374450	&	--	&	0.7851406881	&	0.7851380014	&	31	&	15.00	\\
			52	&	0.7841711020	&	0.7841233641~\cite{pinter2020}	&	0.7842037903	&	0.7841777999	&	147	&	15.30	&	114	&	0.7851465680	&	--	&	0.7851496430	&	0.7851470912	&	30	&	15.08	\\
			54	&	0.7842614465	&	0.7842192995~\cite{pinter2020}	&	0.7842906181	&	0.7842674002	&	138	&	15.93	&	116	&	0.7851552203	&	--	&	0.7851581386	&	0.7851557133	&	29	&	15.91	\\
			56	&	0.7843421691	&	0.7843044654~\cite{pinter2020}	&	0.7843683109	&	0.7843474820	&	129	&	16.23	&	118	&	0.7851634339	&	--	&	0.7851662060	&	0.7851639008	&	29	&	16.91	\\
			58	&	0.7844145892	&	0.7843807534~\cite{pinter2020}	&	0.7844381066	&	0.7844193486	&	121	&	15.83	&	120	&	0.7851712379	&	--	&	0.7851738734	&	0.7851716778	&	28	&	16.30	\\
			60	&	0.7844798073	&	0.7844492943~\cite{pinter2020}	&	0.7845010402	&	0.7844840764	&	110	&	16.07	&	122	&	0.7851786591	&	--	&	0.7851811668	&	0.7851790738	&	27	&	17.02	\\
			62	&	0.7845387477	&	0.7845111362~\cite{pinter2020}	&	0.7845579827	&	0.7845425956	&	103	&	15.93	&	124	&	0.7851857221	&	--	&	0.7851881101	&	0.7851861131	&	26	&	17.13	\\
			64	&	0.7845921910	&	0.7834620877~\cite{pinter2020}	&	0.7846096710	&	0.7845956664	&	95	&	15.71	&	126	&	0.7851924497	&	--	&	0.7851947255	&	0.7851928210	&	26	&	16.63	\\
			66	&	0.7846408000	&	0.7845910589~\cite{pinter2020}	&	0.7846567322	&	0.7846439513	&	89	&	15.20	&	128	&	0.7851988626	&	--	&	0.7852010332	&	0.7851992129	&	25	&	16.36	\\
			\bottomrule
		\end{tabular}
	}
\end{table}

\begin{figure}
	\centering
	\subfloat[$(\hat{\geo{P}}_{16},0.771861)$]{
	\begin{tikzpicture}[scale=5]
	\draw[dashed] (0,0) -- (0.2163,0.0539) -- (0.3801,0.1794) -- (0.4793,0.3573) -- (0.5000,0.5595) -- (0.4390,0.7532) -- (0.3070,0.9060) -- (0.1320,0.9912) -- (0,1) -- (-0.1320,0.9912) -- (-0.3070,0.9060) -- (-0.4390,0.7532) -- (-0.5000,0.5595) -- (-0.4793,0.3573) -- (-0.3801,0.1794) -- (-0.2163,0.0539) -- cycle;
	\draw (0,0)--(0,1);
	\draw (0,0)--(0.1320,0.9912);\draw (0,0)--(-0.1320,0.9912);
	\draw (0.2163,0.0539)--(-0.1320,0.9912);\draw (0.2163,0.0539)--(-0.3070,0.9060);
	\draw (0.3801,0.1794)--(-0.3070,0.9060);\draw (0.3801,0.1794)--(-0.4390,0.7532);
	\draw (0.4793,0.3573)--(-0.4390,0.7532);\draw (0.4793,0.3573)--(-0.5000,0.5595);
	\draw (0.5000,0.5595)--(-0.5000,0.5595);\draw (0.5000,0.5595)--(-0.4793,0.3573);
	\draw (0.4390,0.7532)--(-0.4793,0.3573);\draw (0.4390,0.7532)--(-0.3801,0.1794);
	\draw (0.3070,0.9060)--(-0.3801,0.1794);\draw (0.3070,0.9060)--(-0.2163,0.0539);
	\draw (0.1320,0.9912)--(-0.2163,0.0539);
	\end{tikzpicture}
}
\subfloat[$(\hat{\geo{P}}_{32},0.782133)$]{
	\begin{tikzpicture}[scale=5]
	\draw[dashed] (0,0) -- (0.1083,0.0131) -- (0.2043,0.0450) -- (0.2910,0.0947) -- (0.3661,0.1606) -- (0.4266,0.2401) -- (0.4702,0.3301) -- (0.4950,0.4271) -- (0.5000,0.5271) -- (0.4850,0.6261) -- (0.4507,0.7200) -- (0.3984,0.8052) -- (0.3302,0.8783) -- (0.2491,0.9363) -- (0.1587,0.9768) -- (0.0654,0.9979) -- (0,1) -- (-0.0654,0.9979) -- (-0.1587,0.9768) -- (-0.2491,0.9363) -- (-0.3302,0.8783) -- (-0.3984,0.8052) -- (-0.4507,0.7200) -- (-0.4850,0.6261) -- (-0.5000,0.5271) -- (-0.4950,0.4271) -- (-0.4702,0.3301) -- (-0.4266,0.2401) -- (-0.3661,0.1606) -- (-0.2910,0.0947) -- (-0.2043,0.0450) -- (-0.1083,0.0131) -- cycle;
	\draw (0,0)--(0,1);
	\draw (0,0)--(0.0654,0.9979);\draw (0,0)--(-0.0654,0.9979);
	\draw (0.1083,0.0131)--(-0.0654,0.9979);\draw (0.1083,0.0131)--(-0.1587,0.9768);
	\draw (0.2043,0.0450)--(-0.1587,0.9768);\draw (0.2043,0.0450)--(-0.2491,0.9363);
	\draw (0.2910,0.0947)--(-0.2491,0.9363);\draw (0.2910,0.0947)--(-0.3302,0.8783);
	\draw (0.3661,0.1606)--(-0.3302,0.8783);\draw (0.3661,0.1606)--(-0.3984,0.8052);
	\draw (0.4266,0.2401)--(-0.3984,0.8052);\draw (0.4266,0.2401)--(-0.4507,0.7200);
	\draw (0.4702,0.3301)--(-0.4507,0.7200);\draw (0.4702,0.3301)--(-0.4850,0.6261);
	\draw (0.4950,0.4271)--(-0.4850,0.6261);\draw (0.4950,0.4271)--(-0.5000,0.5271);
	\draw (0.5000,0.5271)--(-0.5000,0.5271);\draw (0.5000,0.5271)--(-0.4950,0.4271);
	\draw (0.4850,0.6261)--(-0.4950,0.4271);\draw (0.4850,0.6261)--(-0.4702,0.3301);
	\draw (0.4507,0.7200)--(-0.4702,0.3301);\draw (0.4507,0.7200)--(-0.4266,0.2401);
	\draw (0.3984,0.8052)--(-0.4266,0.2401);\draw (0.3984,0.8052)--(-0.3661,0.1606);
	\draw (0.3302,0.8783)--(-0.3661,0.1606);\draw (0.3302,0.8783)--(-0.2910,0.0947);
	\draw (0.2491,0.9363)--(-0.2910,0.0947);\draw (0.2491,0.9363)--(-0.2043,0.0450);
	\draw (0.1587,0.9768)--(-0.2043,0.0450);\draw (0.1587,0.9768)--(-0.1083,0.0131);
	\draw (0.0654,0.9979)--(-0.1083,0.0131);
	\end{tikzpicture}
}
\subfloat[$(\hat{\geo{P}}_{64},0.784596)$]{
	\begin{tikzpicture}[scale=5]
	\draw[dashed] (0,0) -- (0.0531,0.0031) -- (0.1018,0.0108) -- (0.1492,0.0231) -- (0.1953,0.0400) -- (0.2398,0.0615) -- (0.2820,0.0874) -- (0.3216,0.1174) -- (0.3581,0.1513) -- (0.3911,0.1887) -- (0.4201,0.2291) -- (0.4451,0.2723) -- (0.4656,0.3178) -- (0.4815,0.3650) -- (0.4926,0.4136) -- (0.4988,0.4631) -- (0.5000,0.5129) -- (0.4963,0.5627) -- (0.4876,0.6118) -- (0.4741,0.6598) -- (0.4559,0.7062) -- (0.4332,0.7506) -- (0.4061,0.7924) -- (0.3750,0.8314) -- (0.3403,0.8670) -- (0.3022,0.8990) -- (0.2612,0.9270) -- (0.2178,0.9507) -- (0.1725,0.9700) -- (0.1257,0.9846) -- (0.0783,0.9944) -- (0.0319,0.9995) -- (0,1) -- (-0.0319,0.9995) -- (-0.0783,0.9944) -- (-0.1257,0.9846) -- (-0.1725,0.9700) -- (-0.2178,0.9507) -- (-0.2612,0.9270) -- (-0.3022,0.8990) -- (-0.3403,0.8670) -- (-0.3750,0.8314) -- (-0.4061,0.7924) -- (-0.4332,0.7506) -- (-0.4559,0.7062) -- (-0.4741,0.6598) -- (-0.4876,0.6118) -- (-0.4963,0.5627) -- (-0.5000,0.5129) -- (-0.4988,0.4631) -- (-0.4926,0.4136) -- (-0.4815,0.3650) -- (-0.4656,0.3178) -- (-0.4451,0.2723) -- (-0.4201,0.2291) -- (-0.3911,0.1887) -- (-0.3581,0.1513) -- (-0.3216,0.1174) -- (-0.2820,0.0874) -- (-0.2398,0.0615) -- (-0.1953,0.0400) -- (-0.1492,0.0231) -- (-0.1018,0.0108) -- (-0.0531,0.0031) -- cycle;
	\draw (0,0)--(0,1);
	\draw (0,0)--(0.0319,0.9995);\draw (0,0)--(-0.0319,0.9995);
	\draw (0.0531,0.0031)--(-0.0319,0.9995);\draw (0.0531,0.0031)--(-0.0783,0.9944);
	\draw (0.1018,0.0108)--(-0.0783,0.9944);\draw (0.1018,0.0108)--(-0.1257,0.9846);
	\draw (0.1492,0.0231)--(-0.1257,0.9846);\draw (0.1492,0.0231)--(-0.1725,0.9700);
	\draw (0.1953,0.0400)--(-0.1725,0.9700);\draw (0.1953,0.0400)--(-0.2178,0.9507);
	\draw (0.2398,0.0615)--(-0.2178,0.9507);\draw (0.2398,0.0615)--(-0.2612,0.9270);
	\draw (0.2820,0.0874)--(-0.2612,0.9270);\draw (0.2820,0.0874)--(-0.3022,0.8990);
	\draw (0.3216,0.1174)--(-0.3022,0.8990);\draw (0.3216,0.1174)--(-0.3403,0.8670);
	\draw (0.3581,0.1513)--(-0.3403,0.8670);\draw (0.3581,0.1513)--(-0.3750,0.8314);
	\draw (0.3911,0.1887)--(-0.3750,0.8314);\draw (0.3911,0.1887)--(-0.4061,0.7924);
	\draw (0.4201,0.2291)--(-0.4061,0.7924);\draw (0.4201,0.2291)--(-0.4332,0.7506);
	\draw (0.4451,0.2723)--(-0.4332,0.7506);\draw (0.4451,0.2723)--(-0.4559,0.7062);
	\draw (0.4656,0.3178)--(-0.4559,0.7062);\draw (0.4656,0.3178)--(-0.4741,0.6598);
	\draw (0.4815,0.3650)--(-0.4741,0.6598);\draw (0.4815,0.3650)--(-0.4876,0.6118);
	\draw (0.4926,0.4136)--(-0.4876,0.6118);\draw (0.4926,0.4136)--(-0.4963,0.5627);
	\draw (0.4988,0.4631)--(-0.4963,0.5627);\draw (0.4988,0.4631)--(-0.5000,0.5129);
	\draw (0.5000,0.5129)--(-0.5000,0.5129);\draw (0.5000,0.5129)--(-0.4988,0.4631);
	\draw (0.4963,0.5627)--(-0.4988,0.4631);\draw (0.4963,0.5627)--(-0.4926,0.4136);
	\draw (0.4876,0.6118)--(-0.4926,0.4136);\draw (0.4876,0.6118)--(-0.4815,0.3650);
	\draw (0.4741,0.6598)--(-0.4815,0.3650);\draw (0.4741,0.6598)--(-0.4656,0.3178);
	\draw (0.4559,0.7062)--(-0.4656,0.3178);\draw (0.4559,0.7062)--(-0.4451,0.2723);
	\draw (0.4332,0.7506)--(-0.4451,0.2723);\draw (0.4332,0.7506)--(-0.4201,0.2291);
	\draw (0.4061,0.7924)--(-0.4201,0.2291);\draw (0.4061,0.7924)--(-0.3911,0.1887);
	\draw (0.3750,0.8314)--(-0.3911,0.1887);\draw (0.3750,0.8314)--(-0.3581,0.1513);
	\draw (0.3403,0.8670)--(-0.3581,0.1513);\draw (0.3403,0.8670)--(-0.3216,0.1174);
	\draw (0.3022,0.8990)--(-0.3216,0.1174);\draw (0.3022,0.8990)--(-0.2820,0.0874);
	\draw (0.2612,0.9270)--(-0.2820,0.0874);\draw (0.2612,0.9270)--(-0.2398,0.0615);
	\draw (0.2178,0.9507)--(-0.2398,0.0615);\draw (0.2178,0.9507)--(-0.1953,0.0400);
	\draw (0.1725,0.9700)--(-0.1953,0.0400);\draw (0.1725,0.9700)--(-0.1492,0.0231);
	\draw (0.1257,0.9846)--(-0.1492,0.0231);\draw (0.1257,0.9846)--(-0.1018,0.0108);
	\draw (0.0783,0.9944)--(-0.1018,0.0108);\draw (0.0783,0.9944)--(-0.0531,0.0031);
	\draw (0.0319,0.9995)--(-0.0531,0.0031);
	\end{tikzpicture}
}
\caption{Three optimal $n$-gons $(\hat{\geo{P}}_n,A(\hat{\geo{P}}_n))$}
\label{figure:Un}
\end{figure}
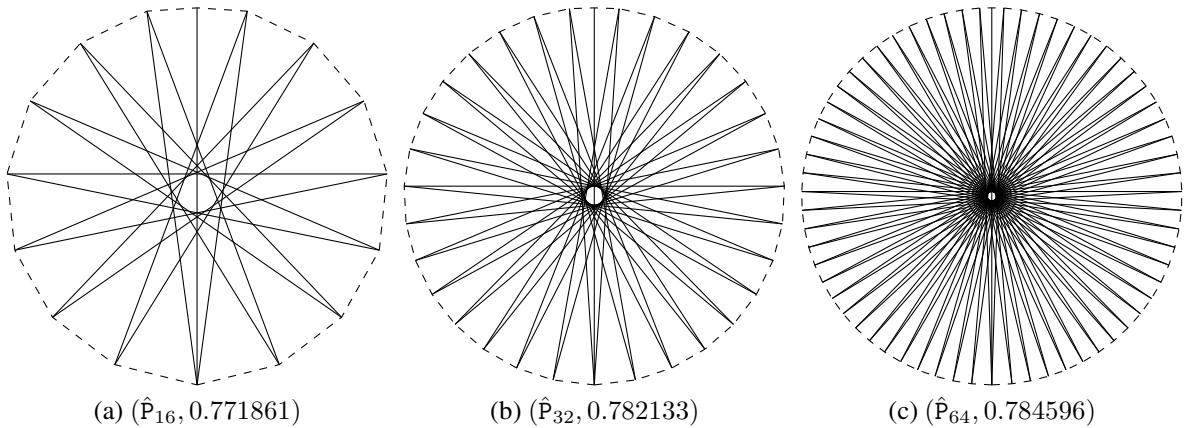

\begin{figure}[h]
	\centering
	\subfloat[$(\geo{R}_5^+,0.672288)$]{
		\begin{tikzpicture}[scale=4]
			\draw[dashed] (0,0) -- (0.5000,0.3633) -- (0.3090,0.9511) -- (0,1) -- (-0.3090,0.9511) -- (-0.5000,0.3633) -- cycle;
			\draw (0,1) -- (0,0) -- (0.3090,0.9511) -- (-0.5000,0.3633) -- (0.5000,0.3633) -- (-0.3090,0.9511) -- (0,0);
		\end{tikzpicture}
	}
\subfloat[$(\geo{P}_6^1,0.674941)$]{
	\begin{tikzpicture}[scale=4]
		\draw[dashed] (0,0) -- (0.5000,0.3975) -- (0.3397,0.9405) -- (0,1) -- (-0.3397,0.9405) -- (-0.5000,0.3975) -- cycle;
		\draw (0,1) -- (0,0) -- (0.3397,0.9405) -- (-0.5000,0.3975) -- (0.5000,0.3975) -- (-0.3397,0.9405) -- (0,0);
	\end{tikzpicture}
}
\subfloat[$(\geo{P}_6^2,0.674981)$]{
	\begin{tikzpicture}[scale=4]
		\draw[dashed] (0,0) -- (0.5000,0.4018) -- (0.3433,0.9392) -- (0,1) -- (-0.3433,0.9392) -- (-0.5000,0.4024) -- cycle;
		\draw (0,1) -- (0,0) -- (0.3433,0.9392) -- (-0.5000,0.4018) -- (0.5000,0.4018) -- (-0.3433,0.9392) -- (0,0);
	\end{tikzpicture}
}\\
	\subfloat[$(\geo{P}_6^3,0.674981)$]{
		\begin{tikzpicture}[scale=4]
			\draw[dashed] (0,0) -- (0.5000,0.4023) -- (0.3437,0.9391) -- (0,1) -- (-0.3437,0.9391) -- (-0.5000,0.4023) -- cycle;
			\draw (0,1) -- (0,0) -- (0.3437,0.9391) -- (-0.5000,0.4023) -- (0.5000,0.4023) -- (-0.3437,0.9391) -- (0,0);
		\end{tikzpicture}
	}
	\subfloat[$(\geo{P}_6^4,0.674981)$]{
		\begin{tikzpicture}[scale=4]
			\draw[dashed] (0,0) -- (0.5000,0.4023) -- (0.3438,0.9391) -- (0,1) -- (-0.3438,0.9391) -- (-0.5000,0.4023) -- cycle;
			\draw (0,1) -- (0,0) -- (0.3438,0.9391) -- (-0.5000,0.4023) -- (0.5000,0.4023) -- (-0.3438,0.9391) -- (0,0);
		\end{tikzpicture}
	}
	\subfloat[$(\geo{P}_6^5,0.674981)$]{
		\begin{tikzpicture}[scale=4]
			\draw[dashed] (0,0) -- (0.5000,0.4024) -- (0.3438,0.9391) -- (0,1) -- (-0.3438,0.9391) -- (-0.5000,0.4024) -- cycle;
			\draw (0,1) -- (0,0) -- (0.3438,0.9391) -- (-0.5000,0.4024) -- (0.5000,0.4024) -- (-0.3438,0.9391) -- (0,0);
		\end{tikzpicture}
	}
	\caption{All $6$-gons $(\geo{P}_6^k,A(\geo{P}_6^k))$ generated by Algorithm~\ref{algo:ccp}}
	\label{figure:ccp:U6}
\end{figure}
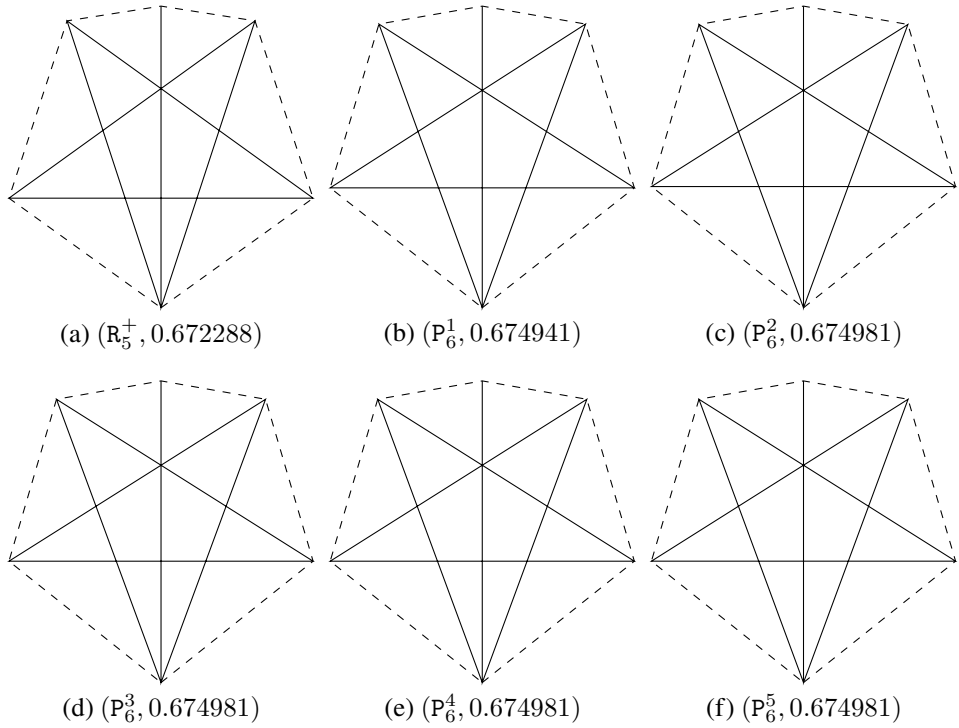

\begin{table}[t]
	\footnotesize
	\centering
	\caption{Vertices of $6$-gons generated by Algorithm~\ref{algo:ccp}}
	\label{table:ccp:U6}
	\resizebox{\linewidth}{!}{
		\begin{tabular}{@{}l|ccccc|l@{}}
			\toprule
			$6$-gon & \multicolumn{5}{c|}{Coordinates $(x_i,y_i)$} & Area \\
			\cmidrule{2-6} & $(x_1,y_1)$ & $(x_2,y_2)$ & $(x_3,y_3)$ &	$(x_4,y_4)$ & $(x_5,y_5)$	&\\
			\midrule
			$\geo{R}_5^+$ & $(0.500000,0.363271)$ & $(0.309017,0.951057)$ & $(0.000000,1.000000)$ &	$(-0.309017,0.951057)$ & $(-0.500000,0.363271)$	&	$0.6722882584$	\\
			$\geo{P}_6^1$ & $(0.500000,0.397446)$ & $(0.339668,0.940545)$ & $(0.000000,1.000000)$ &	$(-0.339668,0.940545)$ & $(-0.500000,0.397446)$	&	$0.6749412362$	\\
			$\geo{P}_6^2$ & $(0.500000,0.401749)$ & $(0.343272,0.939236)$ & $(0.000000,1.000000)$ &	$(-0.343272,0.939236)$ & $(-0.500000,0.401749)$	&	$0.6749808399$	\\
			$\geo{P}_6^3$ & $(0.500000,0.402267)$ & $(0.343702,0.939079)$ & $(0.000000,1.000000)$ &	$(-0.343702,0.939079)$ & $(-0.500000,0.402267)$	&	$0.6749814292$	\\
			$\geo{P}_6^4$ & $(0.500000,0.402329)$ & $(0.343753,0.939060)$ & $(0.000000,1.000000)$ &	$(-0.343753,0.939060)$ & $(-0.500000,0.402329)$	&	$0.6749814401$	\\
			$\geo{P}_6^5$ & $(0.500000,0.402336)$ & $(0.343760,0.939058)$ & $(0.000000,1.000000)$ &	$(-0.343760,0.939058)$ & $(-0.500000,0.402336)$	&	$0.6749814405$	\\
			\bottomrule
		\end{tabular}
	}
\end{table}

\section{Conclusion}\label{sec:conclusion}
We proposed a sequential convex optimization approach to find the largest small $n$-gon for a given even number $n\ge 6$, which is formulated as a nonconvex quadratically constrained quadratic optimization problem. The algorithm, also known as the concave-convex procedure, guarantees convergence to a locally optimal solution.

Without assuming Graham's conjecture nor the existence of an axis of symmetry in our quadratic formulation, numerical experiments on polygons with up to $n=128$ sides showed that each optimal $n$-gon obtained with the algorithm proposed verifies both conditions within the limitation of the numerical computations. Futhermore, for even $6\le n\le 12$, the $n$-gons obtained correspond to the known largest small $n$-gons.


\section*{Acknowledgements}
The author thanks Charles Audet, Professor at Polytechnique Montreal, for helpful discussions on largest small polygons and helpful comments on early drafts of this paper.

\bibliographystyle{ieeetr}
\bibliography{../../research}

\end{document}